\documentclass[11pt]{amsart}
\usepackage{amsmath, amsrefs, amsthm, amssymb}
\usepackage{graphicx}
\usepackage[active]{srcltx}		
\usepackage{pdfsync}					
\newcommand{\dd}{\,{\rm d}}

\newcommand\A{{\mathbb{A}}}
\newcommand\R{{\mathbb{R}}}

\newcommand\N{{\mathbb{N}}}

\newcommand\Z{{\mathbb{Z}}}

\newtheorem{theorem}{Theorem}[section]
\newtheorem{proposition}[theorem]{Proposition}
\newtheorem{lemma}[theorem]{Lemma}

\theoremstyle{definition}
\newtheorem{definition}[theorem]{Definition}
\newtheorem{example}[theorem]{Example}

\theoremstyle{remark}
\newtheorem{remark}[theorem]{Remark}

\numberwithin{equation}{section}
\numberwithin{figure}{section}

\begin{document}

\title[Characterization of sharp algebraic decay]
{Characterization of solutions to dissipative systems with sharp algebraic decay}

\author{Lorenzo Brandolese}

\address{L. Brandolese: Universit\'e de Lyon, Universit\'e Lyon 1,
CNRS UMR 5208 Institut Camille Jordan,
43 bd. du 11 novembre,
Villeurbanne Cedex F-69622, France.}
\email{brandolese{@}math.univ-lyon1.fr}
\urladdr{http://math.univ-lyon1.fr/$\sim$brandolese}

\thanks{Supported by the ANR project DYFICOLTI ANR-13-BS01-0003-01}




\date{\today}

\begin{abstract}
We characterize the set of functions $u_0\in L^2(\R^n)$ such that the solution
of the problem $u_t=\mathcal{L}u$ in $\R^n\times(0,\infty)$ starting from $u_0$
satisfy upper and lower bounds of the form $c(1+t)^{-\gamma}\le \|u(t)\|_2\le c'(1+t)^{-\gamma}$.
Here $\mathcal{L}$ is in a large class of linear pseudo-differential operator 
with homogeneous symbol (including the Laplacian, the fractional Laplacian, etc.).
Applications to nonlinear PDEs will be discussed: in particular our characterization provides necessary and sufficient conditions
on $u_0$ for a solution of the Navier--Stokes system to satisfy sharp upper-lower decay estimates as above.  

In doing so, we will revisit and improve the theory of \emph{decay characters} by C.~Bjorland, C.~Niche, and M.E.~Schonbek,
by getting advantage of the insight provided by the Littlewood--Paley analysis and the use of Besov spaces.
\end{abstract}

\maketitle

\section{Introduction}

The theory of \emph{decay characters} was first introduced by C.~Bjorland and M.E.~Schonbek in~\cite{BjoS09}
and further developped by C.~Niche and M.E.~Schonbek in~\cite{NicS15},
with the motivation of obtaining sharp upper and lower bound estimates for the $L^2$-norm of solutions to a large class of linear 
or semilinear parabolic systems: so far, this theory has been successfully applied, {\it e.g.\/}, 
to the heat equation, the Navier--Stokes equations \cite{BjoS09}, the quasi-geostrophic equations~\cite{NicS15},
several compressible approximations of Navier--Stokes~\cites{NicS15,NicSxx} and to the Navier--Stokes--Voigt equation~\cite{Nicxx}.

The first issue of the present paper is that a slight modification of the original definition of \emph{decay character}
 makes the theory more powerful and more widely applicable: our modification allows in particular to 
 get \emph{necessary and sufficient} conditions on~$u_0$ for the validity of $L^2$-estimates of the form
\begin{equation}
\label{est:ulo}
 c(1+t)^{-\gamma}\le \|e^{t\mathcal{L}}u_0\|_{L^2(\R^n)}\le c'(1+t)^{-\gamma},
\end{equation}
(where $\mathcal{L}$ is a suitable linear pseudo-differential operator, such as the Laplacian, a fractional Laplacian, etc.),
whereas the original approach \cites{BjoS09,NicS15} only gave \emph{sufficient} conditions on~$u_0$
for the validity of~\eqref{est:ulo}.
We also characterize the class of initial data~$u_0$ such that estimates~\eqref{est:ulo} hold in terms of suitable subsets of Besov spaces.

Applications of our analysis to nonlinear problems include the complete characterization of  divergence-free vector fields~$v_0\in L^2(\R^3)^3$
such that the corresponding weak solutions of the Navier--Stokes equations
satisfy two-side bounds for the energy of the form
\begin{equation}
\label{est:vlo}
 c(1+t)^{-\gamma}\le \|v(t)\|_{L^2(\R^n)}\le c'(1+t)^{-\gamma}, \qquad 0<\gamma<5/4.
\end{equation}
One of our main results, Theorem~\ref{th:4equiv}, will provide three different equivalent conditions on~$v_0$ for the vality of~\eqref{est:vlo}.
Our characterization does not go through when $\gamma=5/4$, because the lower bound
of the nonlinear problem in this borderline case is no longer driven by the corresponding lower bound
for heat kernel.

The other important results of this paper will be summarized by Eq.~\eqref{cara3} in the last section.


\subsection*{Notation}
We will often use the symbols $\lesssim$ or $\gtrsim$ in chain of inequalities to avoid the proliferation of different 
constants. For example, 
writing $f(t)\lesssim g(t)$ we mean that $f(t)\le Cg(t)$ for some constant $C>0$ independent on~$t$. 
When we have both $f(t)\lesssim g(t)$ and $g(t)\lesssim f(t)$ we will often write $f(t)\simeq g(t)$.

\section{Revisiting the theory of \emph{decay characters}}

\subsection{Improvement of the basic definitions}

The authors in \cite{BjoS09} introduced the notion of~\emph{Decay indicator}. As the original definition
looks somewhat too restrictive, we redefine it in the following way:

\begin{definition}
Let  $u_0 \in L^2(\R^n), \;B_\rho=\{\xi\in \R^n\colon|\xi| \leq \rho\}$.
The \emph{lower and upper decay indicators} of $u_0$ are the two lower and upper limits
\[
P_r(u_0)_- = \liminf_{\rho \to 0^+} \rho ^{-2r-n} \int _{B_\rho} |\widehat{u}_0 (\xi) |^2 \, d \xi
\qquad
\text{and}
\qquad
P_r(u_0)_+ = \limsup_{\rho \to 0^+} \rho ^{-2r-n} \int _{B_\rho} |\widehat{u}_0 (\xi) |^2 \, d \xi.
\]
When $P_r(u_0)_-=P_r(u_0)_+$, then we can define the \emph{decay indicator} of~$u_0$ as 
$P_r(u_0)=P_r(u_0)_-=P_r(u_0)_+$\,.
\end{definition}

This definition is interesting only for $r \in \left(- \frac{n}{2}, \infty \right)$, as the decay indicator
is always zero when $r\le -n/2$.

\begin{remark}
Originally, the decay indicator was defined in \cite{BjoS09} as the limit
\begin{equation}
\label{eq:dcs}
P_r(u_0) = \lim _{\rho \to 0^+} \rho ^{-2r-n} \int _{B_\rho} |\widehat{u}_0 (\xi) |^2 \, d \xi,
\end{equation}
implicitly assuming that the limit does exist. But this is not always the case for $u_0\in L^2(\R^n)$.
When the above limit does not exist, the lower and upper decay indicators 
are convenient substitutes for obtaining some relevant estimates, as we will see below.
Of course, $P_r(u_0)_-$ and $P_r(u_0)_+$ are always well defined in $[0,+\infty]$.
An example of $v_0\in L^2(\R^n)$ such that 
$P_r(v_0)$ is not well defined by the limit~\eqref{eq:dcs} can be  
constructed putting fast oscillations for $|\widehat v_0(\xi)|$ near the origin. See the following Example~\ref{ex:v0}.

On the other hand, if $u_0\in L^2(\R^n)$ is such that $\widehat u_0(\xi)\sim |\xi|^r$ as $|\xi|\to0$, with $r\in(-n/2,\infty)$,
then $P_r(u_0)_-=P_r(u_0)_+=r$ and so $P_r(u_0)$ is well defined in this case and $P_r(u_0)=r$ as well.
\end{remark}

\begin{definition}
The {\em upper and lower decay characters} of~$u_0\in L^2(\R^n)$ are respectively defined by
\begin{align}
 \label{r+}
 r(u_0)_+&=\sup\{r\in\R\colon P_r(u_0)_+<\infty\},  \\
 \label{r-}
 r(u_0)_-&=\inf\{r\in\R\colon P_r(u_0)_->0\}.
\end{align}
\end{definition}

\begin{remark}
 For any $u_0\in L^2(\R^n)$, the upper and lower decay characters are always well defined (adopting the usual convention that
 $\inf\emptyset=+\infty$) and satisfy the inequality
 \begin{equation}
 \label{ordering}
 -n/2\le r(u_0)_+ \le r(u_0)_-\le \infty.
 \end{equation}
Indeed, for all $r\le -n/2$, we have $P_r(u_0)_+=0$, so the first inequality is immediate.
Let us now prove the inequality in the middle of~\eqref{ordering}.
If, by contradiction, $r(u_0)_-<r(u_0)_+$, then we could choose two real numbers $r<r'$
such that $P_r(u_0)_->0$ and $P_{r'}(u_0)_+<\infty$.
But the last inequality implies 
$P_r(u_0)_+=\limsup_{\rho\to0} \rho^{2(r'-r)}
 \rho ^{-2r'-n} \int _{B_\rho} |\widehat{u}_0 (\xi) |^2 \, d \xi=0$.
On the other hand, $0\le P_{r}(u_0)_-\le P_{r}(u_0)_+=0$, that contredics the inequality
$P_r(u_0)_->0$.

%
\end{remark}
 
 The upper decay character $r(u_0)_+$
 will play a role in obtaining decay estimates from above for $e^{t\Delta}u_0$, and 
 the lower decay character $r(u_0)_+$ will be useful for obtaining estimates from below.
 This is the reason of our terminology. 

Depending on $u_0\in L^2(\R^n)$, the $\sup$ and the $\inf$ in the definition of $r(u_0)_+$ and $r(u_0)_-$ can be achieved or not.
This observation motivates next definition. 
Our definition below is an improvement of that of~\cite{BjoS09}.

\begin{definition}
\label{def:dc}
If $u_0\in L^2(\R^n)$ is such that there exists $r^*\in(-n/2,\infty)$ such that
\begin{equation}
\label{myr}
r^*(u_0)=\max\{r\in\R\colon P_r(u_0)_+<\infty\}=\min\{r\in\R\colon P_r(u_0)_->0\}.
\end{equation}
then we call this number $r^*=r^*(u_0)$ 
the \emph{decay character} of $u_0$.
We define also the decay character of $u_0$ in the two limit situations as follows:
\[
\begin{aligned}
&r^*(u_0)=+\infty, \qquad &\text{if $r(u_0)_+=r(u_0)_-=+\infty$},\\
&r^*(u_0)=-n/2, \qquad &\text{if $r(u_0)_+=r(u_0)_-=-n/2$}.
\end{aligned}
\]
\end{definition}

\begin{remark}
\label{rem:exist}
For $u_0\in L^2(\R^n)$ the decay character $r^*(u_0)$ does not always exist. One reason for this is that
it can happen that $r(u_0)_+<r(u_0)_-$. See Example~\ref{ex:w0}.
Another reason is that one could have $r(u_0)_+=r(u_0)_-$, but the supremum or the infimum appearing in
the definition of $r(u_0)_+$ and $r(u_0)_-$ are nor achieved.
See Example~\ref{ex:u0}.

In fact, $r^*(u_0)$ does exist and belongs to $(-n/2,\infty)$
if and only if there exists $r^*\in(-n/2,\infty)$ such that $0<P_{r^*}(u_0)_-\le P_{r^*}(u_0)_+<\infty$.
The ``only if'' part of this claim is clear.
The ``if'' part holds because, if $r^*<r'$ then  $P_{r'}(u_0)_+=\infty$ (otherwise we would get the contradiction
$P_{r^*}(u_0)_+=0$, as already observed right after~\eqref{ordering}) and so $r^*=\max\{r\in\R\colon P_r(u_0)_+<\infty\}$.
In the same way one sees that $r^*=\min\{r\in\R\colon P_r(u_0)_->0\}$.
\end{remark}

\begin{remark}
Let us illustrate the difference between our improved definition of decay character and the original definition
in~\cite{BjoS09}. 
The main motivation of such an improvement is the validity of the second assertion of Theorem~\ref{th:BS} below.
Our discussion concerns here only the case $r^*(u_0)\in (-n/2,\infty)$, but it could be adapted also
to the limit situations $r^*(u_0)=+\infty$~or~ $-n/2$.

Originally, the decay character $r^*(u_0)$ was defined in \cite{BjoS09}
implicitly assuming that, for all~$r\in\R$, $P_r(u_0)_-=P_r(u_0)_+=P_r(u_0)$ and also assuming that, 
for some $r^*\in(-n/2,\infty)$,
$0<P_{r^*}(u_0)<\infty$. 
Under such two conditions, we see that
$r^*=\max\{r\in\R\colon 0< P_r(u_0)<\infty\}=\min\{r\in\R\colon 0< P_r(u_0)<\infty\}$.
For this reason, if $u_0$ admits a decay character in the sense of~\cite{BjoS09} then $u_0$ admits
a decay character in the sense of our definition and these are the same.
The converse is not true, so the original definition in \cite{BjoS09} is indeed more restrictive than ours:
for example, the function~$v_0\in L^2(\R^n)$, constructed in Example~\ref{ex:v0} admits a decay character
\emph{only} in the sense of our definition.
\end{remark}

\subsection{Applications to upper and lower decay estimates}
\label{sec:appli}

We start with an application of the notions of upper and lower decay indicators.
As in~\cite{NicS15}, we consider the class of (matricial) pseudo-differential operators~$\mathcal{L}$ 
with symbol 
\[ 
\mathcal{M}(\xi)= P(\xi)^{-1}D(\xi)P(\xi),\qquad \text{for a.e. $\xi\in\R^n$},
\]
where $D(\xi)$ and $P(\xi)$ are respectively diagonal and orthogonal matrices of order~$m$, with
$D(\xi)_{ij}=-c_i|\xi|^{2\alpha}\delta_{i,j}$, and $c_i\ge c>0$, for all $i=1,\ldots,m$ and $\alpha>0$.
We also assume that $P(\xi)_{ij}$ are homogeneous functions smooth outside $\xi=0$.

We are interested in establishing $L^2$-estimates from above and below for solutions of the linear problem
\begin{equation}
 \label{lin:pro}
 \begin{cases}
u_t=\mathcal{L} u, & t>0, \;x\in\R^n,\\
u|_{t=0}=u_0,
\end{cases}
\end{equation}
where $u=(u_1,\ldots,u_m)$ and $u_0=(u_{0,1},\ldots,u_{0,m})$.
Basic examples include the heat equation $v_t=\Delta v$ (in this case $P(\xi)=I_m$ and $D(\xi)=-|\xi|^2I_m$) or the evolution problem for the fractional Laplacian ($P(\xi)=I_m$ and $D(\xi)=-|\xi|^{2\alpha}I_m$). 
Examples of physical interest with $P\not=I_m$ arise, \emph{e.g.}, in fluid mechanics, see~\cite{NicS15}.

We will typically assume $u_0\in (L^2(\R^n))^m$. However, from now on we will not distinguish in our notations
between scalar and vector-valued function spaces and write abusively $u_0\in L^2(\R^n)$ also in the vector-valued case.

\begin{proposition}
\label{prop:ulb}
 Let $u_0\in L^2(\R^n)$, let $\mathcal L$ be a pseudo-differential operator as above.
 Let $u$ be the solution of the linear problem~\eqref{lin:pro}.
 \begin{itemize}
 \item[-]
  If $P_r(u_0)_->0$ then there is a constant $C_1>0$ such that, for all $t>0$,
 \[
 C_1(1+t)^{-\frac{1}{\alpha}(r+n/2)}\le \|u(t)\|_2^2.
 \phantom{\le C_2(1+t)^{-\frac1\alpha(r+n/2)}}.
 \]
 \item[-]
 If $P_r(u_0)_+<\infty$, then there is a constant $C_2>0$ such that, for all $t>0$,
 \[
 \phantom{C_1(1+t)^{-\frac{1}{\alpha}(r+n/2)}\le}
  \|u(t)\|_2^2\le C_2(1+t)^{-\frac1\alpha(r+n/2)}.
 \]
 \end{itemize}
 \end{proposition}

 \begin{proof}
 The proof below is a minor modification to that of~\cite{BjoS09}*{Theorem~5.7} (for the particular case $\mathcal{L}=\Delta$)
 or that of~\cite{NicS15}*{Theorem~2.10} (for the general case). 

We have
\[
|e^{\mathcal{M}(\xi)t}\widehat u_0(\xi)|=
|P(\xi)^{-1}e^{D(\xi)t}P(\xi)\widehat u_0(\xi)|
\gtrsim e^{-c|\xi|^{2\alpha}t}|\widehat u_0(\xi)|.
\]
Hence, for any function $\rho=\rho(t)$,
\[
\begin{split}
 \|u(t)\|_2^2
 &\ge \int_{|\xi|\le \rho(t)} |e^{\mathcal{M}(\xi)t} \widehat u_0(\xi)|^2\dd \xi\\
 &\gtrsim  \int_{|\xi|\le \rho(t)} e^{-2c|\xi|^{2\alpha}t}|\widehat u_0(\xi)|^2 \dd\xi
 \ge e^{-2ct\rho(t)^{2\alpha}}\int_{|\xi|\le \rho(t)} |\widehat u_0(\xi)|^2 \dd\xi.
\end{split}
\]
Let $\Phi_r(\rho)=\rho^{-2r-n}\int_{|\xi|\le \rho}|\widehat u_0(\xi)|^2\dd\xi$.
Under the assumption of the first assertion and by the definition of $\liminf$,  we have
$0<P_r(u_0)_-=\lim_{\epsilon\to0^+}\inf_{\rho\in (0,\epsilon]}\Phi_r(\rho)$.
Hence, for some $c_0,\rho_0>0$ and all $0<\rho\le\rho_0$,
we have $\Phi_r(\rho)>c_0$.
The choice $\rho(t)=\rho_0(1+t)^{-1/(2\alpha)}$ then
leads to
\[
 \|u(t)\|_2^2 \gtrsim 
 e^{-2ct\rho(t)^{2\alpha}}
 \rho(t)^{2r+n} \Phi_r(\rho(t))
 \gtrsim \rho(t)^{2r+n}
 \]
and the first claim follows.

Let us prove the second assertion, following again the steps of~\cite{NicS15}.
First of all, the assumption on the symbol $\mathcal{M}(\xi)$ allows us to write
\[
 \begin{split}
 \frac12
  \frac{\dd }{\dd t}\|u(t)\|_2^2
  &=\langle \widehat u,P^{-1}DP\widehat u\rangle  = -\langle (-D)^{-1}P\widehat u,(-D)^{-1}P\widehat u \rangle\\
  &=-\int|(-D)^{-1/2} P\widehat u|^2\dd \xi.
 \end{split}
\]
This insures the 
validity of the energy inequality
\[
\frac{\dd }{\dd t}\|u(t)\|_2^2\le -C\int |\xi|^{2\alpha}|\widehat u(\xi,t)|^2\dd\xi.
\]
The classical Fourier splitting idea \cite{Sch93} is then used to
deduce the estimate
\[
\frac{\dd }{\dd t}\|u\|_2^2 + \rho(t)^{2\alpha}\|u(t)\|_2^2
\lesssim \rho(t)^{2\alpha}\int_{|\xi|\le \rho(t)}|\widehat u(\xi,t)|^2\dd\xi.
\]
Our assumption now reads 
$P_r(u_0)_+=\lim_{\epsilon\to0^+}\sup_{\rho\in(0,\epsilon]}\Phi_r(\rho)<\infty$.
It implies that for some $\rho_0,C>0$ and all $0<\rho\le\rho_0$,
we have
$\Phi_r(\rho)=\rho^{-2r-n}\int_{|\xi|\le \rho(t)}|\widehat u_0(\xi)|^2\dd\xi\le C$.
Hence,
\[
\rho(t)^{2\alpha}\int_{|\xi|\le \rho(t)}|\widehat u(\xi,t)|^2\dd\xi
\lesssim \rho(t)^{2\alpha+2r+n}.
\]
Combining the two last estimates, choosing now $\rho(t)^{2\alpha}=M(1+t)^{-1}$ with $M>r+n/2$, and multiplying by 
the integrating factor $(1+t)^M$ we arrive at
\[
\frac{\dd }{\dd t}\bigl((1+t)^M\|u(t)\|_2^2\bigr)\lesssim (1+t)^{M-1-(2r+n)/(2\alpha)}.
\]
The upper bound follows by integration.
\end{proof}

The following theorem clearly illustrates the importance of the notion of decay character.

\begin{theorem}
\mbox{}\\
\label{th:BS}
\begin{enumerate}
\item
\label{item:t1}
Let~$u_0\in L^2(\R^n)$ be such that the decay character $r^*=r^*(u_0)\in (-n/2,+\infty)$ does exist.
Let~$\mathcal{L}$ as in Proposition~\ref{prop:ulb}, and~$u$ be
the solution of the problem~\eqref{lin:pro}.
Then for some $C_1,C_2>0$, and all positive~$t$,
\begin{equation}
\label{eq:twoside}
 C_1(1+t)^{-\frac{1}{\alpha}(r^*+n/2)}\le \|u(t)\|_2^2 \le C_2(1+t)^{-\frac{1}{\alpha}(r^*+n/2)}.
\end{equation}
\item
\label{item:t2}
Conversely, if  $u_0\in L^2(\R^n)$ is such that 
the solution~$u$ of the problem~\eqref{lin:pro} satisfies estimates~\eqref{eq:twoside} with $r^*\in(-n/2,\infty)$
then $u_0$ possess a decay character and $r^*(u_0)=r^*$.
\end{enumerate}
\end{theorem}

Assertion~\eqref{item:t1} appears also in~\cites{BjoS09,NicS15}, but with their more restrictive definition of $r^*(u_0)$.
This first assertion is an immediate consequence of the estimates obtained in Proposition~\ref{prop:ulb}. 
On the other hand, the validity of assertion~\eqref{item:t2} is made possible
by the fact that, compared to~\cites{BjoS09,NicS15}, in our definition 
we relaxed a little bit the requirements for the existence of the decay character.
The proof of the converse part of Theorem~\ref{th:BS} will be postponed in Remark~\ref{re:pro}, after the characterization of decay characters in terms
of subsets of Besov spaces.


\section{Besov space approach to two-side bounds for  linear dissipative systems}

Let us recall the definition of the homogeneous Besov spaces {\it via\/} the Littlewood-Paley analysis.
Let $\varphi$ be a smooth radial function with support contained in the annulus 
$\{\xi\in \R^n\colon 3/4\le|\xi|\le8/3\}$, such that $\sum_{j\in\Z} \varphi(\xi/2^{-j})=1$ for all 
$\xi\in\R^n\backslash\{0\}$.
Let $\Delta_j$ be the usual Littlewood--Paley localization operator around the frequency $|\xi|\simeq 2^j$, $j\in\Z$, namely, $\widehat{\Delta_j f}=\varphi(\cdot/2^{j})\widehat f$,
see~\cite{BahCD11}.
In this paper we will only need to consider the case $s< n/p$, $1\le p,q\le +\infty$. In this case, the elements of $\dot B^{s,q}_p$ 
can be realized as tempered distributions:

\[
 \dot B^{s}_{p,q}= \Bigl\{ f\in  \mathcal{S}'(\R^n)\colon
 f=\sum_{j\in\Z}\Delta_j f, \;\text{the series being convergent in}\;  \mathcal{S}'(\R^n),\;  \bigl(2^{js}\|\Delta_j f\|_p\bigr)_j \in \ell^q(\Z)\Bigr\},
 \]
normed by
\[
 \qquad
\| f\|_{\dot B^{s}_{p,q}}=\Bigl\|  \bigl(2^{js}\|\Delta_j f\|_p \bigr) \Bigr\|_{\ell^q(\Z)}.
\]

The classical characterization of Besov spaces with negative regularity
 in terms of the heat kernel reads (see, \emph{e.g.}, \cite{BahCD11}*{Theorem~3.4}):
 \begin{equation}
 \label{eq:usub}
  f\in{\dot B}^{-2\sigma}_{p,q} \iff t^\sigma\|e^{t\Delta}f\|_p \in L^q(\R^+,\dd t/t), \qquad \sigma>0,\quad1\le p,q\le \infty,
 \end{equation}
with equivalence of the corresponding norms.
In particular, for $u_0\in L^2(\R^n)$, we have
$u_0\in \dot B^{-2\sigma}_{2,\infty}$ if and only if $\|e^{t\Delta}u_0\|_2\lesssim (1+t)^{-\sigma}$.

\subsection{Two useful subsets of Besov spaces}
\label{sec:subset}
%
We now introduce two subsets of~$\dot B^{-\sigma}_{2,\infty}$:
\[
\dot{\A}^{-\sigma}_{2,\infty}
=\left\{ f\in \dot B^{-\sigma}_{2,\infty} 
\colon 
 \exists\, c,C,M>0\;\text{s.t.}
 \biggl\{
\begin{array}{ll}
 \\
 \|\Delta_j f\|_2\le C2^{\sigma j}  &\forall j\in \Z\\
 \|\Delta_j f\|_2\ge c2^{\sigma j} &\text{for at least one $j\in\Z$ in}\\
 &\text{any interval of lenght~$M$}
 \end{array}
 \right\}
\]
and
\[
\begin{split}
\dot{\mathcal{A}}^{-\sigma}_{2,\infty}
=\biggl\{ f\in   \dot B^{-\sigma}_{2,\infty} 
\colon 
 &\exists\, c,C>0, \;\exists (j_k)_{k\in\N}\subset\Z\;\text{s.t.}\; j_k\to-\infty,\\
 &(j_k-j_{k+1})\in\ell^\infty(\N),\;\text{and}\;
 \biggl\{
\begin{array}{ll}
 \|\Delta_j f\|_2\le C2^{\sigma j}  &\forall j\in \Z\\
 \|\Delta_{j_k} f\|_2\ge c2^{\sigma j_k} &\forall k\in\N
 \end{array}
 \biggr\}.
 \end{split}
\]
These sets are not closed under summation and do not have a linear structure.
We point out the inclusions
\[
\dot{\A}^{-\sigma}_{2,\infty}\subset\dot{\mathcal{A}}^{-\sigma}_{2,\infty}
\subset
\dot B^{-\sigma}_{2,\infty}.\]

It is worth observing that the intersection $L^2(\R^n)\cap \dot{\A}^{-\sigma}_{2,\infty}$ is empty, since dyadic blocks of 
$L^2$ functions satisfy $\|\Delta_j f\|_2\le \|f\|_2$ for all $j\in\Z$.
On the other hand, $\dot{\mathcal{A}}^{-\sigma}_{2,\infty}$ has a nontrivial intersection with $L^2(\R^n)$,
which we will characterize below.

%

\subsection{Two-side bounds for the heat kernel}

Our first applications of the sets $\dot{\A}^{-2\sigma}_{2,\infty}$ and $\dot{\mathcal{A}}^{-2\sigma}_{2,\infty}$
is the following theorem.

\begin{theorem}
\label{th:cara}
 Let $\sigma>0$ and $f$ be a tempered distribution.
 \begin{enumerate}
 \item
 Then,
 $f\in \dot\A^{-2\sigma}_{2,\infty}$ if and only if
 \begin{equation}
  \label{eq:heat-lu}
  \exists\, c_1,c_2>0\quad\text{such that}\quad
  c_1 \,t ^{-\sigma}\le \|e^{t\Delta} f\|_2 \le c_2 \, t^{-\sigma}, \qquad\text{for all $t>0$.}
 \end{equation}
\item
 Moreover,
$f\in \dot{\mathcal{A}}^{-2\sigma}_{2,\infty}$ if and only if
\begin{equation}
  \label{eq:heat-liminf}
   t^\sigma\,\| e^{t\Delta}f\|_2 \in L^\infty(\R^+), \quad\text{and}\quad 
  \liminf_{t\to+\infty}\,t^\sigma\,\|e^{t\Delta}f\|_2>0.
 \end{equation}
 \item In addition,
 $f\in  L^2(\R^n)\cap \dot{\mathcal{A}}^{-2\sigma}_{2,\infty}$ if and only if
 \begin{equation} 
 \label{eq:heat-lub}
 \exists\, c,c'>0\quad\text{such that}\quad
  c \,(1+t)^{-\sigma}\le \|e^{t\Delta} f\|_2 \le c' \, (1+t)^{-\sigma}, \qquad\text{for all $t>0$.}
 \end{equation}
 \end{enumerate}
 \end{theorem}

\begin{proof}
 The property $t^\sigma\,\| e^{t\Delta}f\|_2 \in L^\infty(\R^+)$  and the upper bound in~\eqref{eq:heat-lu},
are equivalent to $f\in \dot B^{-2\sigma}_{2,\infty}$, as noticed in~\eqref{eq:usub}.
In the same way, the upper bound in~\eqref{eq:heat-lub} is equivalent to the condition 
$f\in L^2(\R^n)\cap\dot B^{-2\sigma}_{2,\infty}$.
For this reason, in the proof of~Theorem~\ref{th:cara} we need to focus only on the lower bound properties.

Let us start with a simple preliminary computation.
For $t>0$, let $p=p(t)\in\Z$ such that $4^p\le t<4^{p+1}$.
Letting $d_j=\|\Delta_j f\|_2^2\, 2^{-4j\sigma}$ we have
(the positive constant $c$ is not the same from line to line):
\begin{equation}
 \label{eq:size}
 \begin{split}
  \| t^\sigma e^{t\Delta}f\|_2^2
  &\simeq \sum_{j\in\Z} t^{2\sigma} e^{-ct 2^{2j}} \|\Delta_j f\|_2^2\\
  &\simeq\sum_{j\in\Z} t^{2\sigma} e^{-ct 4^{j}}2^{4j\sigma}d_j\\
  &\simeq\sum_{j\in\Z} 4^{2j\sigma} e^{-c\,4^j} d_{j-p},
 \end{split}
\end{equation}
where in the last step was obtained by shifting the summation index.

Assume now  
$f\in \dot\A^{-2\sigma}_{2,\infty}$. Then, for some constant  $M>0$ as in the definition of~$\dot \A^{-2\sigma}_{2,\infty}$,
 and some constant $c>0$ independent on~$p$ we have ($c$ is not the same from line to line),
\[
\begin{split}
 \| t^\sigma e^{t\Delta}f\|_2^2  
 &\gtrsim  \sum_{|j|\le M} 4^{2j\sigma} e^{-c\,4^j} d_{j-p} \\
 &\gtrsim \max_{|j|\le M}{d_{j-p}} \gtrsim 1,
 \end{split}
\]
where in the last inequality we used 
 $\max_{|j|\le M} 2^{-4\sigma (j-p)}\|\Delta_{j-p} f\|_2^2\ge c$, that follows from the definition 
 of~$\dot{\A}^{-2\sigma}_{2,\infty}$.
 This proves the lower bound in~\eqref{eq:heat-lu}.

 If we assume instead the weaker condition $f\in\dot{\mathcal{A}}^{-2\sigma}_{2,\infty}$,
 then take $M=\| (j_k-j_{k+1})\|_{\ell^\infty}$, where $j_k\to-\infty$ is as in the definition
  of~$\dot{\mathcal{A}}^{-2\sigma}_{2,\infty}$.
 Then, there exist an integer $p_0$ and $c>0$ such that for all $p\ge p_0$ we have  
 $\max_{|j|\le M}{d_{j-p}}=\max_{|j|\le M} 2^{-4\sigma (j-p)}\|\Delta_{j-p} f\|_2^2\ge c>0$.
 As $c$ is independent on~$p$, and hence on $t$, at least in some interval $[t_0,+\infty)$, 
 we get now
 \[
\begin{split}
 \| t^\sigma e^{t\Delta}f\|_2^2  
 \gtrsim 1 \quad \text{on $[t_0,\infty)$}.
 \end{split}
\]
This in turn gives
\[
\liminf_{t\to+\infty}\,t^\sigma\,\|e^{t\Delta}f\|_2>0.
\]

Conversely, if we assume~\eqref{eq:heat-lu}, then  the lower bound in~\eqref{eq:heat-lu}  and  the computation 
at the beginning of the proof imply
\begin{equation}
\label{ababb}
\sum_{j\in\Z} 4^{2j\sigma} e^{-c\,4^j} d_{j-p} \ge  c'>0,
\end{equation}
with $4^p\simeq t$ and $c'$ independent on~$p$.
On the other hand, the upper bound~\eqref{eq:heat-lu} imply that $u_0\in \dot B^{-2\sigma}_{2,\infty}$, hence
$(d_j)\in \ell^\infty(\Z)$.
As $(4^{2j\sigma}e^{-c\,4^j})\in \ell^1(\Z)$, we can find $M>0$ such that 
\[ \sum_{|j|\ge M} 4^{2j\sigma}e^{-c\,4^j}<c'/(2\|d_j\|_{\ell^\infty}).\]
Combining the two last estimates we get
\[
 \begin{split}
 \sum_{|j|\le M} 4^{2j\sigma}e^{-c4^j}d_{j-p} \ge c'/2.
 \end{split}
\]
On the other hand,
$
 \sum_{|j|\le M} 4^{2j\sigma}e^{-c4^j}d_{j-p} \le C \max_{|j|\le M} d_{j-p}
$
for some $C>0$ independent on~$p$.
It then follows that, for some $c>0$ independent on~$p$,
\[
\max_{|j|\le M} d_{j-p} \ge c.
\]
This insures that $f\in  \dot\A^{-2\sigma}_{2,\infty}$.

If, instead of estimates~\eqref{eq:heat-lu} we have the weaker conditions~\eqref{eq:heat-liminf}, then~\eqref{ababb} holds true
only for all $p\ge p_0$, for some $p_0\in\Z$ large enough.
Applying the above argument to $p=p_0+kM$, $k\in\N$, 
we get, for some $c>0$ independent on~$k$,
\[
\max_{|j|\le M} d_{j-(p_0+kM)} \ge c.
\]
Therefore, we can then construct a sequence of integers $(j_k)$ such that
$j_k\to-\infty$, $|j_{k+1}-j_k|\le 3M$, and $d_{j_k}\ge c>0$ for all $k\in\N$.
This implies $f\in  \dot{\mathcal{A}}^{-2\sigma}_{2,\infty}$.

It only remains to prove the last claim of Theorem~\ref{th:cara}.
If $f\in L^2(\R^n)\cap \dot{\mathcal{A}}^{-2\sigma}_{2,\infty}$ then $\|e^{t\Delta}f\|_2^2\le \|f\|_2^2$ and 
the upper bound estimate~\eqref{eq:heat-lub}
immediately follows from the first of~\eqref{eq:heat-liminf}. 
The lower-bound estimate in~\eqref{eq:heat-lub} follows 
combining the fact that $\|e^{t\Delta}f\|_2^2\ge ct^{-\sigma}$, valid in $[t_0,+\infty)$ for some large enough $t_0>0$,
by the second of~\eqref{eq:heat-liminf}, with the inequality 
$\|e^{t\Delta}f\|_2^2\ge \|e^{t_0\Delta}f\|_2^2$, valid in the interval $[0,t_0]$.

Conversely, if estimates~\eqref{eq:heat-lub} hold then 
$f\in  \dot{\mathcal{A}}^{-2\sigma}_{2,\infty}$ by the previous part of the Theorem. Moreover, by Fatou's lemma, 
$\|f\|_2^2\le \liminf_{t\to0^+} \bigl(\int  e^{-2t|\xi|^2}f(\xi) \bigr)\dd \xi \le (c')^2<\infty$, so $f\in L^2(\R^n)$.
\end{proof}

\subsection{A Besov space characterization through $e^{t\mathcal{L}}$}

In order establish the natural generalization of Theorem~\ref{th:cara} to the solutions of the linear problem~\eqref{lin:pro},
we first need to extend the classical heat kernel characterization  of Besov space~\eqref{eq:usub} to the setting of the operator
$e^{t\mathcal{L}}$.
Specifically, we establish the following theorem.

\begin{theorem}
 \label{BCD:th2.24}
Let $\sigma>0$, $1\le p,q\le \infty$. Then for any $u\in\mathcal{S}'(\R^n)$ such that the Littlewood--Paley series 
$u=\sum_{j\in\Z} \Delta_j u$ converges in~$\mathcal{S}'(\R^n)$,
we have 
$f\in\dot B^{-2\sigma}_{p,q}(\R^n)$ if and only if $\|t^{\sigma/\alpha}e^{t\mathcal{L}}f\|_p \in L^q(\R^+,\dd t/t)$.
Moreover, the two norms
\[
\|f\|_{\dot B^{-2\sigma}_{p,q}}\qquad\text{and}\qquad
\Bigl\|\,\|t^{\sigma/\alpha}e^{t\mathcal{L}}f\|_p\Bigr\|_{L^q(\R^+,\dd t/t)}
\]
are equivalent.
\end{theorem}

\begin{proof}

The proof is an adaptation of that in~\cite{BahCD11}*{Theorem~2.34} for the special case $\mathcal{L}=\Delta$.
Recall that
\[
\mathcal{F}(e^{t\mathcal{L}}f)(\xi)=P(\xi)^{-1}e^{tD(\xi)}P(\xi) \widehat f(\xi),
\]
where $D(\xi)=-|\xi|^{2\alpha}D$, where $D=\mbox{diag}(c_1,\ldots,c_n)$, with $\alpha,c_j>0$,
and $P(\xi)$ is an orthogonal matrix. The entries $P(\xi)_{j,k}$ $(j,k=1,\ldots,m)$ are assumed to be homogeneous functions
smooth for $\xi\not=0$.

\begin{lemma}
 \label{BCD:lemma2.4}
 Let $p\in[1,\infty]$, and $f$ be a tempered distribution with support in the annulus $\lambda\mathcal{C}$, 
 where $\mathcal{C}=\{\xi\in\R^n\colon 0<r_1<|\xi|<r_2\}$ and $\lambda>0$.
Then, for all $t>0$.
\[
\|e^{t\mathcal{L}}f\|_p \le Ce^{-ct\lambda^{2\alpha}}\|f\|_p,
\]
for some $c,C>0$ independent on $\lambda$ and $t$.
\end{lemma}

\begin{proof}
The operator $e^{t\mathcal{L}}$ is a convolution operator of the form $e^{t\mathcal{L}}f=G(\cdot,t)*f$.
By the homogeneity assumptions on $P$ and $D$, the symbol
$\widehat{G}(\xi,t)=P(\xi)^{-1}e^{tD(\xi)}P(\xi)$ satisfies the scaling relation
$\widehat G(\xi,t)=\widehat G(\xi t^{1/(2\alpha)},1)$, and so 
 $G(\lambda x,t)=\lambda^{-n}G(x,t\lambda^{-2\alpha})$.

Replacing $\widehat f$ with its dilate $\widehat f(\lambda\cdot)$ we can reduce to the case $\lambda=1$.
So assume now $\lambda=1$ and consider a smooth cut-off function $\phi(\xi)$ identically equal to one on~$\mathcal{C}$ with
compact support bounded away from $\xi=0$.
As $f=\phi*f$, we have
\[
\begin{split}
 \|e^{t\mathcal{L}}f\|_p=\|G(\cdot,t)*\phi\|_1 \|f\|_p,
\end{split}
\]
and it only remains to show that $\| G(\cdot,t)*\phi\|_1\le Ce^{-ct}$.

Indeed, computing the inverse Fourier transform and integrating by parts,
\[
\begin{split}
 G(\cdot,t)*\phi(x)&=(2\pi)^{-n}\int e^{i\xi\cdot x}\phi(\xi)P(\xi)^{-1}e^{tD(\xi)}P(\xi)\dd\xi\\
&=(1+|x|^2)^{-n}\int \bigl((I-\Delta_\xi)^ne^{i\xi\cdot x}\bigr)\phi(\xi)P(\xi)^{-1}e^{tD(\xi)}P(\xi)\dd\xi\\
&=(1+|x|^2)^{-n}\int e^{i\xi\cdot x} \bigl((I-\Delta_\xi)^n\bigl(\phi(\xi)P(\xi)^{-1}e^{tD(\xi)}P(\xi)\bigr)\dd\xi\\
\end{split}
\]
But $(I-\Delta_\xi)^n\bigl(\phi(\xi)P(\xi)^{-1}e^{tD(\xi)}P(\xi)\bigr)$
can be bounded by a linear combination of functions of the form $|\psi_\beta(\xi)| \, |\partial^\beta_\xi e^{tD(\xi)}|$,
with $\psi_\beta(\xi)$ smooth with support contained in that of $\phi(\xi)$,  
$\beta\in\N^n$ and $|\beta|\le 2n$.
Hence, for all $\xi$, 
\[
\bigl|(I-\Delta_\xi)^n\bigl(\phi(\xi)P(\xi)^{-1}e^{-tD(\xi)}P(\xi)\bigr|\le C(1+t)^{2d} e^{-ct|\xi|^{2\alpha}}\chi(\xi),
\]
for some smooth $\chi$ with compact support contained in that of~$\phi$. 
It follows, 
\[
\|G(\cdot,t)*\phi\|_1\le Ce^{-ct}
\]
and the assertion of the Lemma follows.
\end{proof}

Now assume that $u\in \dot B^{-2\sigma}_{p,q}$.
Then, by Lemma~\ref{BCD:lemma2.4},
\begin{equation}
\label{esir}
\| t^{\sigma/\alpha}e^{t\mathcal{L}}f\|_p
\le\sum_{j\in\Z} t^{\sigma/\alpha} \|\Delta_j e^{t\mathcal{L}}f\|_p
\lesssim \sum_{j\in\Z} t^{\sigma/\alpha}4^{j\sigma}e^{-ct\, 4^{j\alpha}}\eta_j,
\end{equation}
with $(\eta_j)\in \ell^q$ and $\|(\eta_j)\|_{\ell^q}=\|f\|_{\dot B^{-2\sigma}_{p,q}}$.
But, for all $\sigma,c,\alpha>0$,
\begin{equation}
 \label{BCD:lemma2.35}
 \sup_{t>0}\sum_{j\in\Z}t^{\sigma/\alpha}4^{j\sigma}e^{-ct\,4^{j\alpha}}<\infty,
\end{equation}
as one can easily check splitting the summation in two terms corresponding to $j$ such that $4^j\le t^{-1/\alpha}$, next $4^j> t^{-1/\alpha}$.

In the case $q=\infty$, we readily deduce
\[ 
\|t^{\sigma/\alpha}e^{t\mathcal{L}}f\|_p \lesssim \|(\eta_j)\|_{\ell^\infty}\le \|f\|_{\dot B^{-2\sigma}_{p,\infty}},
\]
that agrees with~\eqref{eq:usub2}.

In the case $1\le q<\infty$,
we need to estimate the integral below first applying~\eqref{esir}, next H\"older inequality to the summation 
with weight $4^{j\sigma}e^{-ct\,4^{j\alpha}}$ and~\eqref{BCD:lemma2.35}, next Fubini's theorem and the definition
of the Gamma function:
\[
\begin{split}
 \int_0^\infty t^{\sigma q/\alpha}\|e^{t\mathcal{L}}f\|_p^q\dd t
 &\lesssim \int_0^t \biggl(\sum_{j\in\Z} t^{\sigma/\alpha} 4^{j\sigma}e^{-ct\, 4^{j\alpha}} \eta_j\biggr)^q\dd t/t\\
 &\lesssim \int_0^t \biggl(\sum_{j\in\Z} t^{\sigma/\alpha} 4^{j\sigma} e^{-ct\, 4^{j\alpha}} \eta_j^q\biggr)\dd t/t\\
 &\lesssim\Gamma(\sigma/\alpha)\sum_{j\in\Z} \eta_j^q 
 \lesssim \|u\|_{\dot B^{-2\sigma}_{p,q}}^q.
\end{split}
\]

\medskip
To prove the reverse estimate,
let us first establish the following identity:
\begin{equation*}
 \label{my:id}
\begin{split}
  \int_0^\infty t^{\sigma/\alpha} &(-D(\xi))^{1+\sigma/\alpha}e^{tD(\xi)} \widehat v(\xi)\dd t\\
  &= \int_0^\infty  t^{1+\sigma/\alpha}|\xi|^{2\alpha(1+\sigma/\alpha)} 
  	\mbox{diag}(c_1^{1+\sigma/\alpha} e^{-c_1t|\xi|^{2\alpha}}  ,\ldots, c_n^{1+\sigma/\alpha} e^{-c_n t|\xi|^{2\alpha}})
	\widehat v(\xi)\dd t/t\\
  &=\int_0^\infty t^{1+\sigma/\alpha}\mbox{diag}(c_1^{1+\sigma/\alpha}e^{-c_1t},\ldots,c_n^{1+\sigma/\alpha}e^{-c_nt})
      \widehat v(\xi)\dd t/t\\
  &=\Gamma(1+\sigma/\alpha)\widehat v(\xi)
\end{split}
\end{equation*}
Applying this to $\widehat v(\xi)=P(\xi)\widehat{\Delta_j f}(\xi)$ we obtain
\[
\begin{split}
\widehat{\Delta_j f}(\xi)
&=\frac{1}{\Gamma(1+\sigma/\alpha)}
\int_0^\infty t^{\sigma/\alpha}P(\xi)^{-1}(- D(\xi))^{1+\sigma/\alpha}e^{tD(\xi)}P(\xi)\widehat{\Delta_j u}(\xi)\dd t\\
&=\frac{1}{\Gamma(1+\sigma/\alpha)}
\int_0^\infty  t^{\sigma/\alpha}P(\xi)^{-1}(-D(\xi))^{1+\sigma/\alpha}P(\xi)\mathcal{F}({e^{t\mathcal{L}}\Delta_j u})(\xi)\dd t\\  
  \end{split}
\]
We now make use of the homogeneity properties of $P(\xi)$ and $D(\xi)$ and 
apply a classical Fourier multiplier theorem (see, \emph{e.g.} \cite{BahCD11}*{Lemma~2.2}),
next the identity $e^{t\mathcal{L}}=e^{t\mathcal{L}/2}e^{t\mathcal{L}/2}$ with Lemma~\ref{BCD:lemma2.4} to get
\begin{equation}
\label{eq:calc5}
\begin{split}
 \|\Delta_j f\|_p
 &\lesssim \int_0^\infty t^{\sigma/\alpha} 4^{\alpha j(1+\sigma/\alpha)} \|e^{t\mathcal{L}}\Delta_j f\|_p\dd t\\
 &\lesssim \int_0^\infty t^{\sigma/\alpha} 4^{\alpha j(1+\sigma/\alpha)} 
 	e^{-ct\,4^{j\alpha}}\|e^{t\mathcal{L}/2}\Delta_j f\|_p\dd t\\
 &\lesssim \int_0^\infty t^{\sigma/\alpha} 4^{\alpha j(1+\sigma/\alpha)} 
 	e^{-2ct\,4^{j\alpha}}\|e^{t\mathcal{L}} f\|_p\dd t\\	
\end{split}
 \end{equation}
When $q=\infty$ we just have to observe that
\begin{equation}
 \|\Delta_j f\|_p\lesssim 2^{2j\sigma}\sup_{t>0}\bigl( t^{\sigma/\alpha}\|e^{t\mathcal{L}}f\|_p\bigr)
\end{equation}
to conclude that 
\[\|f\|_{\dot{B}^{-2\sigma}_{p,\infty}}
\lesssim \sup_{t>0}\bigl( t^{\sigma/\alpha}\|e^{t\mathcal{L}}f\|_p\bigr) .
\]
When $1\le q<\infty$, we argue as follow:
\[
\begin{split}
 \|f\|_{\dot B^{-2\sigma}_{p,q}}^q
 = \sum_{j\in\Z} 2^{-2jq\sigma}\|\Delta_j f\|_p^q 
 &\lesssim \sum_{j\in\Z} 4^{-jq\sigma}
   	\biggl(
	\int_0^\infty t^{\sigma/\alpha} 4^{\alpha j(1+\sigma/\alpha)} e^{-ct\, 4^{j\alpha}}\|e^{t\mathcal{L}}f\|_p\dd t				\biggr)^q\\
 &\lesssim\sum_{j\in\Z} 4^{\alpha jq}\biggl(\int_0^\infty e^{-ct\, 4^{j\alpha}}\dd t\biggr)^{q-1}
 	\biggl(\int_0^\infty t^{q\sigma/\alpha} \|e^{t\mathcal{L}}f\|_p^q e^{-ct\, 4^{j\alpha}} \dd t\biggr)\\
 &\lesssim \sum_{j\in\Z} 4^{j\alpha}	
 	\biggl(\int_0^\infty t^{q\sigma/\alpha} \|e^{t\mathcal{L}}f\|_p^q e^{-ct\, 4^{j\alpha}} \dd t\biggr)\\
 &\int_0^\infty\lesssim \biggl(\sum_{j\in\Z} t 4^{\alpha j}e^{-ct\, 4^{j\alpha}}\biggr)  
 	t^{q\sigma/\alpha} \|e^{t\mathcal{L}}f\|_p^q \dd t/t\\
 &\lesssim \int_0^\infty	 t^{q\sigma/\alpha} \|e^{t\mathcal{L}}f\|_p^q \dd t/t\\
\end{split}
\]
Here we used first~\eqref{eq:calc5}, next H\"older inequality, and in the last inequality a particular 
case of Eq.~\eqref{BCD:lemma2.35}.
The proof of Theorem~\ref{BCD:th2.24} is now complete.
\end{proof}

We are now in a position of generalizing~Theorem~\ref{th:cara} to the case of the operator~$\mathcal{L}$.

\begin{theorem}
 \label{th:cara2}
Let $\sigma>0$ and $f$ be a tempered distribution.
 Let also $\mathcal{L}$ a pseudo-differential operator 
 as at the beginning of Section~\ref{sec:appli}
\begin{enumerate}
\item  \label{ite:1}
 Then,
 $f\in \dot\A^{-2\sigma}_{2,\infty}$ if and only if
 \begin{equation}
  \label{eq:heat-lu2}
  \exists c_1,c_2>0\quad\text{such that}\quad
  c_1 \,t ^{-\sigma/\alpha}\le \|e^{t\mathcal{L}}(t)\|_2 \le c_2 \, t^{-\sigma/\alpha}, \qquad\text{for all $t>0$.}
 \end{equation}
\item\label{ite:2}
 Moreover,
$f\in \dot{\mathcal{A}}^{-2\sigma}_{2,\infty}$ if and only if
\begin{equation}
  \label{eq:heat-liminf2}
   t^{\sigma/\alpha}\,\| e^{t\mathcal{L}}(t)\|_2 \in L^\infty(\R^+), \quad\text{and}\quad 
  \liminf_{t\to+\infty}\,t^{\sigma/\alpha}\,\|e^{t\mathcal{L}}f\|_2>0,
 \end{equation}
 \item \label{ite:3} 
 In addition,
 $f\in  L^2(\R^n)\cap \dot{\mathcal{A}}^{-2\sigma}_{2,\infty}$ if and only if
 \begin{equation} 
 \label{eq:heat-lub2}
 \exists c,c'>0\quad\text{such that}\quad
  c \,(1+t)^{-\sigma/\alpha}\le \|e^{t\mathcal{L}}(t)\|_2 \le c' \, (1+t)^{-\sigma/\alpha}, \qquad\text{for all $t>0$.}
 \end{equation}
 \end{enumerate}
\end{theorem}

\begin{proof}

Recalling that the multiplication by $P(\xi)$ and $P(\xi)^{-1}$ conserve the euclidean norm of a vector
and that $D(\xi)$ and $e^{tD(\xi)}$ are diagonal matrices, we can write
\[
\begin{split}
 \|e^{t\mathcal{L}}f\|_2^2
 &=\int|P(\xi)^{-1}e^{D(\xi)t}P(\xi)\widehat f(\xi)|^2\dd\xi\\
 &\simeq \int|e^{-ct|\xi|^{2\alpha}}\widehat f(\xi)|^2\dd\xi,
\end{split}
\]
By our previous theorem, we have the analogue of the characterization to~\eqref{eq:usub} for the operator $\mathcal{L}$,
namely, 
\begin{equation}
 \label{eq:usub2}
  f\in{\dot B}^{-2\sigma}_{p,q} \iff t^{\sigma/\alpha}\|e^{t\mathcal{L}}f\|_2 \in L^q(\R^+,\dd t/t),
 \end{equation}
with the equivalence of the respective norms. 

We now argue as in \eqref{eq:size}: We set as before $d_j=\|\Delta_j f\|_2^2\, 2^{-4j\sigma}$.
For $t>0$, let $p=p(t)\in\Z$ such that $4^{p\alpha}\le t<4^{(p+1)\alpha}$. 
Then we have 
\begin{equation}
 \label{eq:size2}
 \begin{split}
  \| t^{\sigma/\alpha} e^{t\mathcal{L}}f\|_2^2
  &\simeq \sum_{j\in\Z} t^{2\sigma/\alpha} e^{-ct 2^{2j\alpha}} \|\Delta_j f\|_2^2\\
  &\simeq\sum_{j\in\Z} t^{2\sigma/\alpha} e^{-ct 4^{j\alpha}}2^{4j\sigma}d_j\\
  &\simeq\sum_{j\in\Z} 4^{2j\sigma} e^{-c\,4^{j\alpha}} d_{j-p},
 \end{split}
\end{equation}
The proof can now be carried on making obvious modification to that of~Theorem~\ref{th:cara}.
\end{proof}


\section{Characterization of $L^2$-functions admitting a decay character}

For any $u_0\in L^2(\R^n)$,  we introduce the quantity
$\sigma(u_0)\in [0,+\infty]$ given by
\begin{equation}
\label{eq:bes-ch}
\sigma(u_0)=\sup\bigl\{\sigma\ge0\colon u_0\in \dot B^{-\sigma}_{2,\infty}\bigr\}.
\end{equation}
Notice that, for $u_0\in L^2(\R^n)$, $\sigma(u_0)$ is always well defined in $[0,+\infty]$.
Indeed, by the Plancherel theorem, $L^2(\R^n)=\dot B^{0}_{2,2}\subset \dot B^{0}_{2,\infty}$.

The relation between the decay character~$r^*(u_0)$ and $\sigma(u_0)$ is given by the following simple result.
\begin{proposition}
\label{prop:obv}
Let $u_0\in L^2(\R^n)$.
When the decay character of $u_0$ does exist, it is given by the formula
(valid also in the limit cases $r^*(u_0)=-n/2$ or $\infty$):
\begin{equation}
\label{eq:relation}
 r^\ast(u_0) +n/2=\sigma(u_0).
\end{equation}
\end{proposition}

\begin{proof}
In the case $r^*(u_0)=-n/2$, we have $+\infty=P_r(u_0)_->0$ for all $r>-n/2$, and so $\|e^{t\Delta}u_0\|_2^2\ge c_\epsilon(1+t)^{-\epsilon}$
for all $\epsilon>0$ by the first part of Proposition~\ref{prop:ulb}.
Applying the characterization~\eqref{eq:usub} of Besov spaces with negative regularity 
in terms of the heat kernel we get $\sigma(u_0)=0$.
In the case $r^*(u_0)=+\infty$, we have $P_r(u_0)_+<\infty$ for all $r\in\R$ and so 
$\|e^{t\Delta}u_0\|_2^2\le C_k(1+t)^{-k}$ for all $k\ge0$
by the second part of Proposition~\ref{prop:ulb}. We apply again~\eqref{eq:usub} to conclude that $\sigma(u_0)=+\infty$.
In the case $-n/2<r^*(u_0)<\infty$ the conclusion follows immediately combining both parts of Proposition~\ref{prop:ulb} with~\eqref{eq:usub}.
\end{proof}

\begin{remark}
C.~Niche and M.E.~Schonbek generalized the
 notion of decay character by introducing, for $s\ge0$ and $u_0\in H^s(\R^n)$, the notation
\[
r^*_s(u_0)=r^*(\Lambda^s u_0), \qquad\text{where $\Lambda=\sqrt{-\Delta}$}.
\]
The condition $u_0\in H^s(\R^n)$ insures that both $u_0$ and $\Lambda^su_0$ belong to $L^2(\R^n)$.
They proved  in~\cite{NicS15}*{Theorem~2.11}
 that, for $u_0\in H^s(\R^n)$, $r^*_s(u_0)=r^*(u_0)+s$ (this formula being valid also in the limit cases).
Proposition~\ref{prop:obv} provides a one-sentence proof for this formula to hold: indeed, $\Lambda^s$ is known to
be an isomorphism between $\dot B^{-\sigma}_{2,\infty}$ and $\dot B^{-(\sigma+s)}_{2,\infty}$ and therefore
$\sigma^*(\Lambda^s u_0)=\sigma^*(u_0)+s$.
\end{remark}

One can ask if it is possible to give a Besov space characterization of the \emph{existence}
of the decay character of an $L^2$-function. Proposition~\ref{prop:3crit} below will provide a positive answer.


\begin{proposition}
\label{prop:3crit}
Let $u_0\in L^2(\R^n)$. 
Then the decay character of $u_0$ does exist and $r^\ast=r^\ast(u_0)\in (-n/2,\infty)$
if and only if there is $\sigma^*>0$ such that 
$u_0\in \dot{\mathcal A}^{-\sigma^*}_{2,\infty}$.
In this case, $\sigma^*=\sigma(u_0)$ and formula~\eqref{eq:relation} holds.
\end{proposition}

\begin{proof}
Let us start assuming
$u_0\in  \dot{\mathcal{A}}^{-\sigma^\ast}_{2,\infty}$, with $\sigma^*\in(0,\infty)$.
For any $\rho>0$, take $j\in \Z$ such that $2^{j-1}<\rho\le2^j$.
Then,
\[
\begin{split}
 \rho^{-2\sigma^*}\int_{|\xi|\le \rho} |\widehat u_0(\xi)|^2\dd\xi
 &\lesssim 2^{-2j\sigma^*} \int_{|\xi|\le 2^j} |\widehat u_0(\xi)|^2\dd\xi\\
 & \lesssim 2^{-2j\sigma^*} \sum_{-\infty<k\le j} \|\Delta_k u_0\|_2^2\\
 & \lesssim 2^{-2j\sigma^*} \Bigl( \sum_{-\infty<k\le j} 2^{2k\sigma^*}\Bigr)\|u_0\|^2_{\dot B^{-\sigma^*}_{2,\infty}}\\
 & \lesssim \|u_0\|^2_{\dot B^{-\sigma^*}_{2,\infty}}.
\end{split}
\]
Letting $r^*+n/2=\sigma^*$, such estimates imply $P_{r^*}(u_0)_+<+\infty$.

Moreover, from our assumption we can take $(j_k)$ as in the definition of 
$\dot{\mathcal{A}}^{-\sigma^\ast}_{2,\infty}$.
Let $\rho_0=2^{j_0}$ and $M=\|(j_k-j_{k+1})\|_{\ell^\infty}$.
For any  $0<\rho< \rho_0$, let $j_k$ be the largest integer of the sequence $(j_k)$ such that $2^{j_k}\le \rho$.
Then we must have $2^{j_k+M}\ge \rho$, otherwise we would find another integer $j_h\in [j_k,j_k+M]$ such that 
$2^{j_h}\le \rho$, thus contredicting the maximality of $j_k$.
Then we have
\begin{equation}
\label{eq:used}
\begin{split}
 \rho^{-2\sigma^*}\int_{|\xi|\le \rho} |\widehat u_0(\xi)|^2\dd\xi
 &\gtrsim 2^{-2(j_k+M)\sigma^*}  \|\Delta_{j_k} u_0\|_2^2 \ge c>0,
\end{split}
\end{equation}
with $c$ independent on~$\rho\in(0,\rho_0)$.
This implies, $P_{r^*}(u_0)_->0$.
By Remark~\ref{rem:exist} we conclude that the decay character does exist and  $r^*(u_0)=r^*\in (-n/2,+\infty)$.

Conversely, assume that the decay character of~$u_0$ does exist and 
$r^*(u_0)\in (-n/2,+\infty)$.
Then we know by the first assertion of Theorem~\ref{th:BS},
 that the solution of the heat equation $e^{t\Delta}u_0$
is bounded from below and above by $C(1+t)^{-\frac12(r^*+n/2)}$. 
We conclude by  the last claim of Theorem~\ref{th:cara} that 
$u_0\in \dot{\mathcal{A}}^{-\sigma^*}_{2,\infty}$ and that $\sigma^*=r^*+n/2$.
\end{proof}

\begin{remark}[Proof of Theorem~\ref{th:BS}]
\label{re:pro}
This first assertion of Theorem~\ref{th:BS} follows immediately from the estimates obtained in Proposition~\ref{prop:ulb}. 
We can now prove the last assertion of Theorem~\ref{th:BS}.
If $u_0\in L^2(\R^n)$ satisfies the estimates~\eqref{eq:twoside},
 then by Theorem~\ref{th:cara2}-\eqref{ite:3}, we obtain $u_0\in \dot{\mathcal{A}}^{-2r^*-n}_{2,\infty}$
and the decay character of~$u_0$ does exist and equals~$r^*$ by Proposition~\ref{prop:3crit}.
Theorem~\ref{th:BS} is now completely proved.
\hfill{$\Box$}
\end{remark}

\section{Application to nonlinear dissipative systems: The Navier--Stokes equation}

As an application of our analysis to nonlinear dissipative systems, let us consider the Navier--Stokes equations in~$\R^3$:
\begin{equation}
 \tag{NS}
 \begin{cases}
  u_t+u\cdot\nabla u+\nabla p=\Delta u, &x\in\R^3,\;t>0\\
  \nabla\cdot u=0 \\
  u(x,0)=u_0(x),
 \end{cases}
\end{equation}
where $u=u(x,t)=(u_1,u_2,u_3)(x,t)$ is the velocity field of an incompressible viscous fluid flow, $p=p(x,t)$ is the pressure, 
and $u_0=(u_{0,1},u_{0,2},u_{0,3})$ is the initial datum.
For any $u_0\in L^2(\R^3)$, with $\nabla\cdot u_0=0$, we know since J.~Leray's classical paper~\cite{Ler34}
that there is at least one solution 
$u\in C_w([0,\infty),L^2(\R^n))\cap L^2_{\text{loc}}(\R^+,H^{1}(\R^3))$
solving~(NS) in the distributional sense and satisfying the strong energy inequality
\begin{equation}
 \label{sei}
 \|u(t)\|_2^2+2\int_s^t\|\nabla u(r)\|_2^2\dd r\le \|u(s)\|_2^2,
\end{equation}
for $s=0$, almost $s>0$ and all $t\ge s$.
The uniqueness and the regularity of such solutions are well known open problems.

The following theorem completely characterize the solutions satisfying sharp two-side decay estimates.
It completes well known Wiegner's Theorem by giving several necessary and sufficient conditions for the validity of bounds from below
for the decay of the energy.
It  improves~\cite{BjoS09}*{Theorem~6.5},
and completes an earlier study by Skal\'ak~\cite{Ska14} where only the upper bounds were discussed.

\begin{theorem}
 \label{th:4equiv}
 Let $u_0\in L^2(\R^3)$ and $\sigma>0$.
 The three following properties are equivalent:
 \begin{itemize}
  \item[(i)]\label{nsi}
  $\displaystyle\liminf_{\rho\to0}\rho^{-2\sigma}\int_{|\xi|\le \rho} |\widehat u_0(\xi)|^2\dd\xi>0 \quad\text{and}\quad
  \displaystyle\limsup_{\rho\to0}\rho^{-2\sigma}\int_{|\xi|\le \rho} |\widehat u_0(\xi)|^2\dd\xi<\infty,
  $
  \item[(ii)]\label{nsii}
$u_0\in\dot{\mathcal{A}}^{-2\sigma}_{2,\infty}$,
  \item[(iii)]\label{nsiii}
  $ (1+t)^{-\sigma}\lesssim \|e^{t\Delta}u_0\|_2\lesssim (1+t)^{-\sigma}$.
 \end{itemize}
If, moreover, $u_0$ is a divergence-free vector-field,  if $u$ is a weak solution of~(NS) as above and $0<\sigma<5/4$, then
the three previous properties are equivalent to
\begin{itemize}
\item[(iv)]\label{nsiv} 
  $(1+t)^{-\sigma}\lesssim\|u(t)\|_2\lesssim (1+t)^{-\sigma}$.
\end{itemize} 
\end{theorem}

\begin{proof}
The equivalence \mbox{(i)$\iff$(ii)} is just a reformulation of Proposition~\ref{prop:3crit}.
Indeed, as explained in~Remark~\ref{rem:exist},
the decay character  $r^*(u_0)$ does exist and belongs to $(-n/2,\infty)$
if and only if there exists $r^*\in(-n/2,\infty)$ such that $0<P_{r^*}(u_0)_-\le P_{r^*}(u_0)_+<\infty$.

The equivalence \mbox{(ii)$\iff$(iii)} was established in Theorem~\ref{th:cara}.

The implication \mbox{(iii)$\,\Rightarrow\,$(iv)} relies on classical Wiegner's theorem~\cite{Wie87}: in the 3D case and in the absence of external forces,
the Theorem in~\cite{Wie87} can restated as follows:  if  $u_0\in L^2(\R^n)$ is a divergence-free
vector field such that $\|e^{t\Delta}u_0\|_2 \lesssim (1+t)^{-\sigma}$, and if $u$ is a weak solution of (NS)
as above, then the difference $w(t)=u(t)-e^{t\Delta}u_0$ satisfies the decay estimates
\begin{equation*}
\|w(t)\|_2\lesssim
\begin{cases}
 (1+t)^{-1/4 -\sigma}, &\text{if $0\le \sigma<1$},\\
 \ln(e+t)(1+t)^{-5/4}, &\text{if $\sigma=1$},\\
 (1+t)^{-5/4}, &\text{if $\sigma>1$}.
\end{cases}
\end{equation*}
Notice in particular that, for $0<\sigma<5/4$, one gets $t^{\sigma}\|w(t)\|_2\to0$ as $t\to\infty$.
Therefore, under the same restrictions on~$\sigma$ and the conditions of item~\mbox{(iii)},
it easily follows that $u=e^{t\Delta}u_0+w$
satisfies the estimates as in~\eqref{nsiv}.

The proof of the implication \mbox{(iv)$\,\Rightarrow\,$(iii)} relies on the so-called \emph{inverse Wiegner's theorem}
recently established by Z.~Skal\'ak~\cite{Ska14}: his results asserts (among other things) that 
if $u$ is a weak solution of the Navier--Stokes equation as above, 
satisfying $\|u(t)\|_2\lesssim (1+t)^{-\sigma}$ (with $0\le \sigma\le 5/4$) then $\|e^{t\Delta}u_0\|_2\lesssim(1+t)^{-\sigma}$.
Hence, Wiegner's theorem applies and decomposing $e^{t\Delta}u_0=u-w$ property~\mbox{(iii)} 
follows in the range $0<\sigma<5/4$.
\end{proof}

In the borderline case $\sigma=5/4$, the implication (iii)$\Rightarrow$(iv) still holds.
But the converse implication is no longer true.
Indeed, it is possible to construct examples of Navier--Stokes flows such that 
$\lim_{t\to+\infty} t^{5/4}\|e^{t\Delta}u_0\|_2=0$ and
$(1+t)^{-5/4}\lesssim\|u(t)\|_2\lesssim (1+t)^{-5/4}$.
This can be achieved as follows: one starts with a divergence-free initial datum $u_0\in L^2(\R^3)$ such that $\widehat u_0$ vanishes in a neighborhood of the origin. This insures a fast (exponential) 
$L^2$-decay of the solution of the heat equation $e^{t\Delta} u_0$. Generically (\emph{i.e.\/} in the absence of special symmetries), the matrix $\int_0^\infty\!\!\int (u\otimes u)(y,s)\,{\rm d}y\,{\rm d}t$
will not be a scalar multiple of the identity matrix. By a theorem of T.~Miyakawa and M.E.~Schonbek~\cite{MiyS01}, one obtains a solution~$u$ of the Navier--Stokes equations
 that satisfies the lower bound estimate
$\|u(t)\|_2\gtrsim(1+t)^{-5/4}$, and, by Wiegner's theorem, 
$\|u(t)\|_2\lesssim(1+t)^{-5/4}$.

\section{Examples}

\begin{example}[Construction of $v_0\in L^2(\R^n)$ such that the limit $P_r(v_0)$ is not well defined]
\label{ex:v0}
Let $r\in(-n/2,\infty)$.
Decompose the unit ball $B_1=\{\xi\colon|\xi|\le 1\}$ in concentric dyadic
annuli: $B_1=\Gamma_0\cup\Gamma_{-1}\cup\Gamma_{-2}\cup\ldots$, with $\Gamma_j=\{\xi\colon 2^{j-1}\le |\xi|\le2^j\}$. 
Then set $\widehat v_0(\xi)=|\xi|^{r}$ on $\Gamma_0\cup\Gamma_{-2}\cup\ldots$ and $\widehat v_0(\xi)=0$ elsewhere.
A direct computation shows that  $0<P_r(v_0)_-<P_r(v_0)_+<\infty$.
Indeed, passing to spherical coordinates in the computation of $P_r(v_0)_+$ and $P_r(v_0)_-$,
\[ 
P_r(v_0)_+
=\omega_n\lim_{j\to-\infty}2^{-2j(2r+n)}
\int_0^\infty  \lambda^{2r+n-1}{\bf 1}_{\Gamma_{2j}\cup\Gamma_{2(j-1)}\cup\ldots} \dd\lambda
\]
and
\[ 
P_r(v_0)_-
=\omega_n\lim_{j\to-\infty}2^{-(2j+1)(2r+n)}
\int_0^\infty\lambda^{2r+n-1}{\bf 1}_{\Gamma_{2j}\cup\Gamma_{2(j-1)}\cup\ldots} \dd\lambda,
\]
where $\omega_n$ is the surface of the unit sphere of~$\R^n$ and ${\bf 1}_A$ denotes the indicator function of the set~$A$.
Both limits could be easily computed, but in fact it is simpler to observe that $\frac{P_r(v_0)_-}{P_r(v_0)_+}=2^{-(2r+n)}<1$
to conclude that $P_r(v_0)$ is not well defined by formula~\eqref{eq:dcs}.
In fact,
$0<P_r(u_0)_-<P_r(u_0)_+<\infty$ and by Remark~\ref{rem:exist}~$v_0$ admits a decay character 
in the sense of Definition~\ref{def:dc} and $r^*(v_0)=r$.
\end{example}

\begin{example}
\label{ex:u0}
It can happen that $r(u_0)_+=r(u_0)_-$, yet the decay character
$r^*(u_0)$ does not exist. 
The example is elementary.
Let $u_0\in L^2(\R^n)$ be such that   $\widehat u_0(\xi)=|\xi|^{r_0}\log|\xi|$ 
in a neighborhood of $\xi=0$, for some $r_0>-n/2$.
Then by a simple computation 
$r(u_0)_+=r(u_0)_-=r_0$.
But the decay character $r^*(u_0)$ is not well defined, 
because the $\sup$ and the $\inf$ in the definition of $r(u_0)_+$ and $r(u_0)_-$
 (see Eq:~\eqref{r+}-\eqref{r-}) are not achieved.
For this specific initial datum, the corresponding solution
of the linear problem~\eqref{lin:pro} satisfies estimates similar to those in~\eqref{eq:twoside},
but with a corrective time-dependent logarithmic factor.
Notice that for this example there is no $r$ such that $0<P_r(u_0)<\infty$ (the existence of
such an~$r$ was needed in the original definition of decay character~\cite{BjoS09}).
\end{example}

\begin{example}
\label{ex:w0}
Let us construct an example of $w_0\in L^2(\R^n)$ verifying 
$r(w_0)_+<r(w_0)_-$. Fix $r\in(-n/2,+\infty)$.
Let $(a_k)$ and $(b_k)$ two real decreasing sequences  such that $0<b_{k+1}\le a_k<b_k$ and $b_k\to0$ for all~$k\in\N$,
and consider two more positive sequences $(\eta_k)$, $(h_k)$ to be chosen later.
We define $w_0$ through its Fourier transform: 
\[
\widehat w_0(\xi)=\sum_{k=0}^\infty h_k{\bf 1}_{a_k\le|\xi|\le b_k}.
\]
The condition insuring $w_0\in L^2(\R^n)$ is $\sum_{k=0}^\infty h_k^2(b_k^n-a_k^n)<\infty$.
It is then convenient to set $\eta_k=h_k^2(b_k^n-a_k^n)$, and the first condition is that $(\eta_k)$ is summable.
Let $\Phi_r(\rho)=\rho^{-2r-n}\int_{|\xi|\le \rho} |\widehat w_0(\xi)|^2\dd\xi$.
Then,
\[
P_r(w_0)_+=
\limsup_{\rho\to0}\Phi_r(\rho)=\lim_{k\to\infty}\Phi_r(b_k)=\lim_{k\to\infty}\omega_n b_k^{-(2r+n)}\sum_{j=k}^\infty \eta_j,
\]
and
\[
P_r(w_0)_-=
\liminf_{\rho\to0}\Phi_{r}(\rho)=\lim_{k\to\infty}\Phi_r(a_k)=\lim_{k\to\infty}\omega_n a_k^{-(2r+n)}\sum_{j={k+1}}^\infty\eta_j,
\]
where $\omega_n$ is the measure of the unit sphere.
Now, we choose, 
\[b_k=2^{-2^k},\qquad a_k^n=b_k^n-b_{k}^{2n}, \qquad \eta_k=b_k^{2r+n}-b_{k+1}^{2r+n}.\]
Then, $\sum_{j=k}^\infty \eta_j=b_k^{2r+n}$.
These choices insure that $w_0\in L^2(\R^n)$ as $(\eta_k)$ is indeed summable and determine $(h_k)$.
Moreover, we get $0<P_r(w_0)_+=\omega_n<\infty$
and we conclude that 
\[r(w_0)_+=r.\]
Let us now study $r(w_0)_-$.
First observe that $\sum_{j=k+1}^\infty \eta_j=b_{k+1}^{2r+n}=(b_k^{2r+n})^2$.
Then, 
\[
 a_k^{-(2r+n)}\sum_{j={k+1}}^\infty\eta_j=
[b_k(1-b_k^n)^{-n}]^{-2r-n} (b_k^{2r+n})^2\to 0\qquad\text{as $k\to\infty$},
\]
so that 
\[
P_r(w_0)_-=\liminf_{k\to\infty} \Phi_r(a_k)=0.
\]
For this reason, $r(w_0)_-\not=r$ and $r^*(w_0)$ does not exist.
In fact $r(w_0)_-$ is given by $r'\in(-n/2,\infty)$, such that the limit
\[
\liminf_{\rho\to0}\Phi_{r'}(\rho)=\lim_{k\to\infty}\Phi_{r'}(a_k)=\omega_n a_k^{-(2r'+n)}\sum_{j={k+1}}^\infty\eta_j,
\]
is a strictly positive real:
we need $2r'+n=2(2r+n)$.
In conclusion,
\[
r(w_0)_-=2r+n/2, \qquad r(w_0)_+=r.
\]
\end{example}


%
%
%
%

\section{Conclusions}

Let $\sigma>0$ and $f\in L^2(\R^n)$.
In~\cites{BjoS09} C.~Bjorland and M.E.~Schonbek, proved that:
\begin{subequations}
\begin{equation}
\label{implia}
0<\lim_{\rho\to0^+}\rho^{-2\sigma}\int_{|\xi|\le \rho} |\widehat f(\xi)|^2\dd\xi<\infty
 \quad\Longrightarrow\quad
(1+t)^{-\sigma}\lesssim \|e^{t\Delta}f\|_2\lesssim (1+t)^{-\sigma},
\end{equation}
where $e^{t\Delta}$ denotes the heat kernel in~$\R^n$ and the symbol $\lesssim$ means that the inequality $\le $ holds up to a multiplicative constant independent on time.
In~\cite{NicS15} C.~Niche and M.E.~Schonbek extended this to a large class of pseudo-differential operator~$\mathcal{L}$,
with homogeneous symbol of degree~$2\alpha$, proving that
\begin{equation}
\label{implib}
 0<\lim_{\rho\to0^+}\rho^{-2\sigma}\int_{|\xi|\le \rho} |\widehat f(\xi)|^2\dd\xi<\infty
 \quad\Longrightarrow\quad
(1+t)^{-\sigma/\alpha}\lesssim \|e^{t\mathcal{L}}f\|_2\lesssim (1+t)^{-\sigma/\alpha}.
\end{equation}
\end{subequations}
In the present paper, we showed that the reverse implications in ~\eqref{implia}-\eqref{implib} do not hold, as the limit on the LHS might not exist.
Next, suitably relaxing the condition on the LHS, we got the following characterization of the class of $L^2$-functions 
satisfying sharp two-side decay estimates. Namely, we established that:
\begin{equation}
\label{cara3}
\begin{cases}
  \displaystyle\liminf_{\rho\to0^+}\rho^{-2\sigma}\int_{|\xi|\le \rho} |\widehat f(\xi)|^2\dd\xi>0 \\
  \displaystyle\limsup_{\rho\to0^+}\rho^{-2\sigma}\int_{|\xi|\le \rho} |\widehat f(\xi)|^2\dd\xi<\infty
\end{cases}
\iff
f\in\dot{\mathcal{A}}^{-2\sigma}_{2,\infty},
\iff
 (1+t)^{-\sigma/\alpha}\lesssim \|e^{t\mathcal{L}}f\|_2\lesssim (1+t)^{-\sigma/\alpha},
\end{equation}
where $\dot{\mathcal{A}}^{-2\sigma}_{2,\infty}$ is the subset of the 
Besov space $\dot B^{-2\sigma}_{2,\infty}(\R^n)$ introduced in Section~\ref{sec:subset}.

Moreover, if $f$ is a divergence-free vector field, then in the special case $n=3$, $\mathcal{L}=\Delta$ and $0<\sigma<5/4$, 
we proved in Theorem~\ref{th:4equiv} that the three previous conditions are in turn equivalent to the two-side energy estimate 
of weak solutions of the Navier--Stokes equation starting from~$f$:
\begin{equation}
 \label{ns:ts}
 (1+t)^{-\sigma}\lesssim \|u(t)\|_2\lesssim (1+t)^{-\sigma}.
\end{equation}

When an $L^2$ function satisfy any of the above equivalent conditions~\eqref{cara3}, the exponent~$\sigma\in(0,\infty)$ is uniquely determined. 
We coin this exponent~$\sigma=\sigma(f)$ the \emph{Besov character} of~$f$. 
It is related to (a suitable improvement of) Bjorland--Schonbek's~\cite{BjoS09} notion of ``decay character'', denoted~$r^*(f)$, by the relation~$\sigma(f)=r^*(f)+n/2$.

\end{document}